\RequirePackage{fix-cm}
\documentclass[smallextended]{svjour3}       % onecolumn (second format)
\smartqed  % flush right qed marks, e.g. at end of proof
\usepackage{graphicx}
\usepackage{amsmath}
\usepackage{algorithm}
\usepackage[noend]{algpseudocode}
\usepackage{comment}

\makeatletter
\def\BState{\State\hskip-\ALG@thistlm}
\makeatother

\begin{document}

\title{Efficient QUBO transformation for Higher Degree Pseudo Boolean Functions%\thanks{Grants or other notes
%about the article that should go on the front page should be
%placed here. General acknowledgments should be placed at the end of the article.}
}
%\subtitle{Do you have a subtitle?\\ If so, write it here}

%\titlerunning{Short form of title}        % if too long for running head

\author{Amit Verma        \and
        Mark Lewis \and %etc.
        Gary Kochenberger
}

%\authorrunning{Short form of author list} % if too long for running head

\institute{Amit Verma, Mark Lewis \at
	Craig School of Business, Missouri Western State University, Saint Joseph, MO, 64507, USA \\
	Tel.: +1(816)-271-4357\\
	Fax: +1(816)-271-4338\\
	\email{averma@missouriwestern.edu}          %  \\
%             \emph{Present address:} of F. Author  %  if needed
           \and
           Gary Kochenberger \at
             College of Business, University of Colorado at Denver, Denver, CO, 80217, USA %\\
             %\email{gary.kochenberger@cudenver.edu}
}

\date{Received: date / Accepted: date}
% The correct dates will be entered by the editor

\maketitle

\begin{abstract}
Quadratic Unconstrained Binary Optimization (QUBO) is recognized as a unifying framework for modeling a wide range of problems. Problems can be solved with commercial solvers customized for solving QUBO and since QUBO have degree two, it is useful to have a method for transforming higher degree pseudo-Boolean problems to QUBO format. The standard transformation approach requires additional auxiliary variables supported by penalty terms for each higher degree term. This paper improves on the existing cubic-to-quadratic transformation approach by minimizing the number of additional variables as well as penalty coefficient. Extensive experimental testing on Max 3-SAT modeled as QUBO shows a near 100\% reduction in the subproblem size used for minimization of the number of auxiliary variables.
       
% \PACS{PACS code1 \and PACS code2 \and more}
% \subclass{MSC code1 \and MSC code2 \and more}
\end{abstract}

\section{Introduction} \label{intro}

Pseudo-Boolean functions appearing as non-linear polynomial expressions are an important part in optimization models of statistical mechanics, reliability theory, satisfiability theory, finance, manufacturing and operations research (\cite{boros2002pseudo}). Recently, Quadratic Unconstrained Binary Optimization (QUBO) models have garnered a lot of attention due to its applicability in quantum annealing. QUBO or equivalent Ising Spin $(+1/-1)$ models have emerged as a unifying framework for modeling a wide variety of combinatorial optimization problems (see \cite{kochenberger2014unconstrained} for details). Note that QUBO involves optimization of a pseudo-Boolean function of degree two. QUBO modeling framework relies on quadratic interaction among binary decision variables in the form of $max \; x^{t}Qx \; = \; max \; \sum_{i=1}^{n} \sum_{j=1}^{n} q_{ij} x_i x_j; \; x_i \in \{0,1\} \; \forall i \in \{1,2,...,n\}$. The diagonal and off-diagonal coefficients of the $Q$ matrix are determined based on the application area. Herein $n$ denotes the size of the QUBO instance. Smaller and sparser QUBO can be handled by mixed-integer linear programming solvers like CPLEX and Gurobi by linearizing the quadratic objective through additional constraints. The standard linearization technique replaces each quadratic term $x_i x_j$ a new variable $y_{ij}$ and includes the constraint $y_{ij} = x_i x_j$ as part of the optimization problem. However, as the number of quadratic terms becomes larger, the number of involved constraints increases, thereby restricting the usability of linear programming solvers. Larger QUBO are efficiently handled directly (without any linearization) by heuristics such as tabu search and path relinking described in \cite{wang2012path,lewis2021qfold}.

Higher degree pseudo-Boolean functions are handled via repeated variable substitutions using the Rosenberg quadratization penalty terms (\cite{rosenberg1975reduction}). More specifically, a product of two binary variables $x_i x_j$ is replaced with an additional binary variable $y_{ij}$ along with a penalty term $-M(x_i x_j - 2 x_i y_{ij} - 2 x_j y_{ij} + 3 y_{ij})$ for a maximization problem. The penalty term enforces the binary variable assignments such that the constraint $x_i x_j = y_{ij}$ is satisfied. This process is repeated until the higher degree pseudo-Boolean function is reduced to a degree two polynomial. Note that the additional $y_{ij}$ variables increases the size and density of the augmented $Q$ matrix, thereby impacting the computational performance of the associated optimization problem. The number of auxiliary variables could be minimized using the integer program proposed in \cite{verma2020optimal}. Various heuristic algorithms based on set covers (\cite{buchheim2008efficient}) and greedy techniques (\cite{rodriguez2018linear})  have also been proposed to identify the minimum number of auxiliary variables added to the original pseudo-Boolean function. We refer the readers to \cite{dattani2019quadratization} for a complete discourse on the various quadratization (process of transforming a higher degree pseudo-Boolean function to quadratic) techniques in discrete optimization. The lower bound of the penalty coefficient $M$ is determined using the nonlinear Constraint Programming (CP) approach described in \cite{verma2020optimal}. Utilizing the lower bound of $M$ (compared to an arbitrary choice of $M$) clearly impacts the computational performance of the underlying optimization problem based on the augmented $Q$ matrix.

In this paper, we reduce the size of the integer program proposed in \cite{verma2020optimal} for determining the minimum number of auxiliary variables. Further, we eliminate the need for a CP approach to determine the lower bound of $M$. Instead, we replace it with a linear expression that is easy to compute. We also demonstrate the applicability of the entire transformation process through modeling the popular Max 3-SAT as a cubic pseudo boolean function and determining the minimal augmented QUBO model.  Note that our goal is to demonstrate the reductions of IP and CP model proposed in \cite{verma2020optimal} through using Max 3-SAT as a testbed. The topic of modeling Max 3-SAT as a QUBO has been previously explored in \cite{choi2011different,m3satqubo,kofler2014penalty,gabor2019assessing,bian2020solving,kruger2020quantum}. For example, the authors in \cite{kofler2014penalty} propose a multi-start approach utilizing a set of elite solutions and adaptive memory schemes to solve the resulting cubic pseudo-Boolean penalty function. The authors in \cite{m3satqubo} utilized a termwise reduction approach to reduce the number of variables in QUBO by 33\% compared to the traditional QUBO transformation of Max 3-SAT through Maximum Independent Set reduction proposed in \cite{lucas2014ising}. The authors in \cite{gabor2019assessing,bian2020solving,kruger2020quantum} addressed various experimental challenges of solving Max 3-SAT on the quantum annealing hardware. Next, we present some preliminaries along with a standard procedure to transform a Max 3-SAT instance into a QUBO.

\section{QUBO transformation of Max 3-SAT} \label{transform}

Boolean satisfiability problem (SAT) is a well-studied combinatorial problem in the field of computer science. It determines the existence of an assignment of binary variables that satisfies a boolean formula. A boolean formula is represented by a set of clauses containing a combination of variables involving AND, OR and NOT operators. The Cook-Levin theorem (\cite{cook1971complexity}) proves that SAT is NP-Complete. Therefore if there exists a deterministic polynomial time algorithm for solving SAT, every NP problem can be solved in polynomial time (\cite{karp1972reducibility}). This result emphasizes the importance of SAT in the computational complexity domain. 3-SAT is a form of SAT where each clause is limited to at most three literals (positive or negated variables). Max 3-SAT is an optimization version of 3-SAT where we are interested in maximizing the number of satisfied clauses. The boolean formula is usually represented in the conjunctive normal form (CNF), wherein each clause is a disjunction of literals and the boolean formula is a conjunction of clauses. For example, $(x_1 \cup x_2 \cup x_3) \cap (\overline{x_1} \cup \overline{x_2} \cup x_3) \cap (\overline{x_1} \cup x_2 \cup \overline{x_3})$ is a 3-SAT instance expressed in the CNF format. Note that both positive literals ($x_i$) and negated literals ($\overline{x_i}$) are part of the clauses, and the boolean formula is the conjunction of the clauses, which consists of at most three literals. The satisfying assignment $x_1 = 0, x_2 = 1$ and $x_3 = 1$ represents a solution of 3-SAT. There is a possibility of finding multiple solutions to a 3-SAT instance. In this case, the solution satisfies all clauses in the boolean formula. Hence, we achieve maximum satisfiability.

Next, we describe a QUBO transformation of an instance of Max 3-SAT based on \cite{kofler2014penalty}. First, we need to design an objective function $g(X)$ for each clause. Each clause in the 3-SAT instance can contain 0,1,2 or 3 negated variables. Each of these cases yield a different objective function for Max 3-SAT since we are interested in maximizing the number of satisfied clauses. The four possibilities are:\\
\\
(i) No negations ($x_i \cup x_j \cup x_k$): $g(X) = x_i + x_j + x_k -x_i x_j- x_i x_k - x_j x_k + x_i x_j x_k$ \\
(ii) One negation ($x_i \cup x_j \cup \overline{x_k}$): $g(X) = 1- x_k + x_i x_k + x_j x_k - x_i x_j x_k $ \\
(iii) Two negations ($x_i \cup \overline{x_j} \cup \overline{x_k}$): $g(X) = 1 - x_j x_k + x_i x_j x_k$ \\
(iii) Three negations ($\overline{x_i} \cup \overline{x_j} \cup \overline{x_k}$): $g(X) = 1 - x_i x_j x_k$ \\
\\
Note that the objective function value is $1$ if the corresponding clause is satisfied; $0$ otherwise. For example, the first objective function $g(X)$ corresponding to ($x_i \cup x_j \cup x_k$) is $1$ for all possible assignments $x_i,x_j$ and $x_k$ except $x_i=0,x_j=0$ and $x_k=0$. Thus the objective function can be interpreted as a reward function. We could easily verify the objective function values using the truth tables for $x_i, x_j$ and $x_k$. In this way, we transform a Max 3-SAT instance with $m$ clauses into an unconstrained cubic binary optimization problem $max \; \sum_{i = 1}^{m} g(X)_i$ involving cubic terms like $x_i x_j x_k$. In order to transform this further into QUBO, we need to transform all cubic terms into quadratic terms using auxiliary variables with the help of Rosenberg quadratization presented in \cite{rosenberg1975reduction}. As discussed in Section 1, when a quadratic subterm $x_i x_j$ is substituted with an auxiliary variable $y_{ij}$, the penalty term  $ - M(x_i x_j - 2 x_i y_{ij} - 2 x_j y_{ij} + 3y_{ij})$ (where $M$ is a large positive number) is added to a maximization problem. 

The approach then is to introduce a set of auxiliary variables $Y$ such that the cubic binary optimization problem is transformed into the following QUBO:
\begin{align} \label{uqp}
	max \; \sum_{i = 1}^{m} g(X,Y)_i - M \sum_{aux \; vars \; y_{ij}} (x_i x_j - 2 x_i y_{ij} - 2 x_j y_{ij} + 3y_{ij})
\end{align}
Since each clause generates exactly one cubic term, the augmented QUBO consists of at most $m$ auxiliary variables. Hence a Max 3-SAT instance with $n$ variables and $m$ clauses would yield a QUBO instance with at most $n+m$ variables. The optimal solution to the QUBO instance corresponds to the number of satisfied clauses ($\le m$) in the Max 3-SAT instance.

In this transformation process, we prefer to use a smaller number of auxiliary variables since the size and density of the resulting QUBO affects its computational performance (\cite{verma2020optimal}). For this purpose, the authors in \cite{verma2020optimal} proposed an integer program to minimize the number of auxiliary variables for higher degree pseudo-Boolean polynomials. We extend their work by significantly reducing the size of the proposed linear program. In this process, some variables are determined apriori and eliminated for consideration in the integer program. Our findings are presented in Section \ref{ipandlb}. The authors in \cite{verma2020optimal} also proposed a non-linear CP model to minimize the penalty coefficient $M$ used in Rosenberg quadratization. We further eliminate the need for the CP program with the help of Equation \ref{rednlp} in Section \ref{ipandlb}. This leads to a significant reduction in overall computational effort. 

Next, we will demonstrate the standard QUBO transformation for Max 3-SAT with the help of an example.

\subsection{Example}
Let us consider the following Max 3-SAT instance with $4$ variables and $4$ clauses in the CNF format:\\
\\
$(x_1 \cup x_2 \cup x_3) \cap (\overline{x_1} \cup \overline{x_2} \cup x_3) \cap (\overline{x_1} \cup x_2 \cup \overline{x_3}) \cap (x_1 \cup x_2 \cup \overline{x_4})$\\
\\
By using the transformations described in Section \ref{intro}, we arrive at the following optimization problem for the Max 3-SAT instance:\\
\\
$max \; (x_1 + x_2 + x_3 -x_1 x_2- x_1 x_3 - x_2 x_3 + x_1 x_2 x_3) + (1 - x_1 x_2 + x_1 x_2 x_3) + (1 - x_1 x_3 + x_1 x_2 x_3) + (1- x_4 + x_1 x_4 + x_2 x_4 - x_1 x_2 x_4)$\\
\\
Summing the individual terms in the objective function yields:\\
\\
$ max \; 3 + x_1 + x_2 + x_3 - x_4 - 2 x_1 x_2 - 2 x_1 x_3 + x_1 x_4 - x_2 x_3 + x_2 x_4 + 3 x_1 x_2 x_3 - x_1 x_2 x_4$\\
\\
To convert to the QUBO framework, we want to transform the cubic terms $x_1 x_2 x_3$ and $x_1 x_2 x_4$ with the lowest number of auxiliary variables using the integer program proposed in \cite{verma2020optimal}. Herein, we utilize only one auxiliary variable $y_{12}$ in lieu of the quadratic subterm $x_1 x_2$. Note that this additional binary variables $y_{12}$ covers both cubic terms. However, the auxiliary variable introduces a penalty term $-M (x_1 x_2 - 2 x_1 y_{12} - 2 x_2 y_{12} + 3 y_{12})$ according to Rosenberg quadratization to enforce the constraint $y_{12} = x_1 x_2$. We could also minimize the value of $M$ using the non-linear program proposed in \cite{verma2020optimal} in order to improve computational performance. However, we utilize $M=10$ for the following QUBO transformation: \\
\\
$ max \; 3 + x_1 + x_2 + x_3 - x_4 - 2 x_1 x_2 - 2 x_1 x_3 + x_1 x_4 - x_2 x_3 + x_2 x_4 + 3 x_3 y_{12} - x_4 y_{12} - 10 (x_1 x_2 - 2 x_1 y_{12} - 2 x_2 y_{12} + 3 y_{12})$\\
\\
which can be further simplified to:\\
$ max \; 3 + x_1 + x_2 + x_3 - x_4 - 30 y_{12} - 12 x_1 x_2 - 2 x_1 x_3 + x_1 x_4 + 20 x_1 y_{12} - x_2 x_3 + x_2 x_4 + 20 x_2 y_{12} + 3 x_3 y_{12} - x_4 y_{12})$\\
\\
This could be easily represented as:\\
$max \; 3 + x^{t} Q x$\\
where the matrix $Q$ in the upper triangular format is given by:\\
$\begin{bmatrix}
1 & -12 & -2 & 1 & 20 \\
0 & 1 & -1 & 1 & 20 \\
0 & 0  & 1  & 0  & 3 \\
0 & 0  & 0  & -1  & -1 \\
0 & 0  & 0  & 0  & -30 
\end{bmatrix}
$\\
\\
Solving this QUBO yields the solution $x_1 = 0, x_2 = 1, x_3 = 1, x_4 = 1$ and $y_{12} = 0$ having an objective function value of $4$ meaning all four clauses are satisfied.

\section{Cubic to Quadratic Transformation} \label{ipandlb}
As discussed in Section \ref{transform}, each 3-SAT clause introduces a cubic term in the objective function. These cubic terms are transformed to quadratic terms using the penalties given by Rosenberg quadratization. In this transformation process, we are interested in minimizing the number of auxiliary variables. We used the integer program proposed in \cite{verma2020optimal} to address this issue. For every cubic term, the integer program identifies the quadratic subterm that will be replaced by a new auxiliary variable. Recall that the transformation $x_i x_j = y_{ij}$ adds a penalty term $-M(x_i x_j - 2 x_i y_{ij} - 2 x_j y_{ij} + 3 y_{ij})$ for a maximization problem.

However, the integer program that identifies the quadratic subterm covering every cubic term needs $3$ variables for each cubic term. Thus, the size of the resulting IP is $O(3m)$, where $m$ is the number of clauses. To reduce the overall size of the IP, we utilize the following lemma:

\begin{lemma}
	In order to minimize the number of auxiliary variables during the transformation, the quadratic subterm $x_i x_j$ will always be utilized if it dominates with respect to the frequency count in all the cubic terms in which it appears i.e. $(f(x_i x_j) > f(x_i x_k)) \; \& \; (f(x_i x_j) > f(x_j x_k)) \; \forall (x_i x_j) \subset (x_i x_j x_k) \; s.t. \; (x_i x_j x_k) \subset C$ where $f()$ represents the function that counts the occurrence of each quadratic subterm across all cubic terms and $C$ denotes the set of cubic terms. 
\end{lemma}

\begin{proof}
	Each cubic term will use exactly one auxiliary variable in lieu of the quadratic subterm. The quadratic subterm $x_i x_j$ should be preferred for replacement by an auxiliary variable compared to other quadratic subterms $x_i x_k$ and $x_j x_k$ as long as $x_i x_j$ is the dominant quadratic subterm in every cubic term it is part of. Since the quadratic subterm $x_i x_j$ is dominant in terms of frequency counts in every shared cubic term, there is no better assignment than replacing $x_i x_j$ with an auxiliary variable. If we assume that $x_i x_k$ is used instead of $x_i x_j$ for variable substitution in the cubic term $x_i x_j x_k$, then replacing the chosen quadratic subterm with $x_i x_j$ will not deteriorate the utilized number of total auxiliary variables. Hence, by contradiction, the optimal set of quadratic subterms must include $x_i x_j$.
\end{proof}

By using Lemma 1, we easily identify the quadratic subterms that are dominating other quadratic subterms in the proposed integer program. In this way, we identify the quadratic subterm that is chosen for coverage for some $k'th$ cubic term. 

Next, we will illustrate Lemma 1 with the help of an example. Let us assume that the cubic terms in the pseudo-Boolean polynomial are $x_1 x_2 x_3, x_1 x_2 x_4,$ $x_1 x_2 x_5, x_1 x_2 x_6, x_2 x_3 x_7$ and $x_2 x_3 x_8$. For the first cubic term, $f(x_1 x_2)=4, f(x_1 x_3)=1, f(x_2 x_3)=3$. Hence $x_1 x_2$ dominates $x_1 x_3$ and $x_2 x_3$ according to Lemma 1. Thus, for the cubic term $x_1 x_2 x_3$, the quadratic subterm $x_1 x_2$ will be replaced by an auxiliary variable $y_{12}$ yielding $x_1 x_2 x_3 = y_{12} x_3$. 

After introducing the least number of auxiliary variables, we add the penalty $-M(x_i x_j - 2 x_i y_{ij} - 2 x_j y_{ij} + 3 y_{ij})$ for every transformation $x_i x_j = y_{ij}$ to the objective function described in Section 1. To facilitate the solver, we are interested in minimizing the magnitude of the penalty coefficient $M$. For this purpose, following non-integer program in \cite{verma2020optimal} is solved using a CP solver:
\begin{align} \label{nlp}
	M_{LB} = max \; \frac{ \sum_{i<j< k} a_{ijk} (x_i x_j x_k - y_{ij} x_k )}{\sum_{i < j} (x_i x_j - 2 x_i y_{ij} - 2 x_j y_{ij} + 3 y_{ij})}
\end{align}
Herein $a_{ijk}$ is the corresponding coefficient of the cubic term $x_i x_j x_k$. We need to maximize the right-hand side of Equation \ref{nlp} with the constraint that the denominator is non-zero to obtain the lower bound on the penalty coefficient $M$. To reduce the computational effort of the CP model using Equation 2, we propose the following reduction of Equation \ref{nlp}:
\begin{align} \label{rednlp}
	M_{LB} = max (\forall_{x_i x_j = y_{ij}} (\sum_{k} a_{ijk} \; , \; - \sum_{k} a_{ijk} ))
\end{align}

Next, we will prove the equivalence of Equation \ref{nlp} and \ref{rednlp}. In order to maximize the right hand side of Equation \ref{nlp}, we need to simultaneously maximize the numerator and minimize the denominator. The denominator of each term of the right hand side of Equation \ref{nlp} takes values 0, 1 or 3. Recall that we are interested in non-zero values of the numerator and denominator. The non-zero cases when denominator is 1 is given by (1) $x_i =1, x_j = 1, y_{ij} = 0$ and $x_k = 1$ (2) $x_i =1 , x_j=0, y_{ij}=1$ and $x_k=1$ (3) $x_i =0 , x_j=1, y_{ij}=1$ and $x_k=1$. Note that other variable assignments yields a numerator or denominator value of zero. Scenario (1) reduces the right hand side to $\sum_k a_{ijk}$ for the transformation $x_i x_j = y_{ij}$, if we consider only transformation with one auxiliary variable $y_{ij}$. If we also consider an additional variable substitution $x_m x_n = y_{mn}$, the right hand side is reduced to $\frac{\sum_k (a_{ijk} + a_{mnk})}{1+1}$ which is the arithmetic mean of the two values $\sum_k a_{ijk}$ and $\sum_k a_{mnk}$. One of these two values is guaranteed to exceed the expression of the arithmetic mean. If we consider three auxiliary variables $y_{ij}, y_{mn}$ and $y_{pq}$ then the right hand side of Equation \ref{nlp} is reduced to $\frac{\sum_k (a_{ijk} + a_{mnk} + a_{pqk} )}{1+1+1}$ which represents the average of three extreme values. Hence, scenario (1) yields an estimate of $M_{LB} = max (\forall_{x_i x_j = y_{ij}} (\sum_{k} a_{ijk}))$. Similarly, scenarios (2) and (3) reduces to $M_{LB} = max (\forall_{x_i x_j = y_{ij}} (-\sum_{k} a_{ijk}))$. We could also combine scenario (1) for some auxiliary variables, for example $y_{ij}$ and $y_{mn}$ while using scenario (2) or (3) for other auxiliary variables, for example $y_{pq}$. Thus the right hand side of Equation \ref{nlp} is reduced to $\frac{\sum_k (a_{ijk} + a_{mnk} - a_{pqk} )}{1+1+1}$ which is guaranteed to be less than one of the following extreme values: (i) $\sum_{k} a_{ijk}$ (ii) $\sum_{k} a_{mnk}$ and greater than (iii) $\sum_{k} a_{pqk}$ by the properties of arithmetic mean. The denominator value of 3 is possible when $x_i = 0, x_j = 0, y_{ij} = 1$ and $x_k = 1$ which reduces the numerator to $\sum_k a_{ijk}$ similar to other scenarios. However, the denominator value is higher. Since we are interested in maximizing the right hand side of Equation \ref{nlp}, this scenario is dominated by others. Therefore Equation \ref{nlp} is equivalent to Equation \ref{rednlp}.

Next, we illustrate an application of Equation \ref{rednlp}. For the example using cubic terms $x_1 x_2 x_3, x_1 x_2 x_4,$ $x_1 x_2 x_5, x_1 x_2 x_6, x_2 x_3 x_7$ and $x_2 x_3 x_8$, the minimum number of auxiliary variables is 2 namely $x_1 x_2 = y_{12}$ and $x_2 x_3 = y_{34}$. In this way, the first term $x_1 x_2 x_3$ is reduced to $y_{12} x_3$. Thus we apply Equation \ref{rednlp} as follows:
\begin{align} \label{rednlpex1}
	M_{LB} = max(1+1+1+1,1+1,-1-1-1-1,-1-1) = 4
\end{align}
because the first auxiliary variable appears in 4 cubic terms while the second auxiliary variable is shared with 2 cubic terms. This result matches the output of the non-linear program proposed in Equation \ref{nlp}.

\section{Solution Technique}
Our pseudo-code is summarized in Algorithm \ref{algo1}. First, we transform all clauses in the Max 3-SAT instance using the appropriate rewards discussed in Section \ref{intro}. These rewards vary with respect to the number of negated variables in each clause. The reward function is designed such that the value is $1$ if the clause is satisfied. The goal is to maximize the number of satisfied clauses in Max 3-SAT. Second, we sum all the reward terms contributed by each clause. Note that this generates cubic terms, some of which might be eliminated. Thereafter, we extract all the cubic terms that are part of the objective function. We are interested in transforming these cubic terms into quadratic terms using a minimum number of auxiliary variables by solving the integer program proposed in \cite{verma2020optimal} while utilizing Lemma 1 in Section \ref{ipandlb}. We will analyze the impact of Lemma 1 on the size of the integer program in Section \ref{tests}. Next, we identify the lower bound on the penalty coefficient $M$ using Equation \ref{rednlp}. Thereafter, we perform the variable substitutions to convert each cubic term in the objective function to a quadratic term with the help of Rosenberg quadratization. This process concludes with the generation of the $Q$ matrix. Lastly, we solve the final quadratic program using the path relinking and tabu search based generic heuristic solver designed in \cite{lewis2021qfold}. We present the results from the heuristic on benchmark SATLIB instances of Max 3-SAT in Section \ref{tests}.

\begin{algorithm}
	\caption{Minimal QUBO for Max 3-SAT}\label{algo1}
	\begin{algorithmic}[1]
		\Procedure{}{}
		\State $n \gets \text{number of variables}$
		\State $m \gets \text{number of clauses}$.
		\State $S \gets \text{set of all terms in the objective function}$.
		\State $C \gets \text{set of all cubic terms in the objective function}$.
		\State $IP \gets \text{integer program to identify the minimum number of auxiliary variables after applying Lemma 1}$.
		\State $LB \gets \text{lower bound of penalty coefficient obtained from Equation 3}$.
		\State $Y \gets \text{set of auxilliary variables}$.
		\BState \emph{top}:
		\State $\text{Convert each clause} \; i \;  \text{into a cubic penalty term} \; g(X) \; \text{described in Section 1}$
		\State $\text{Objective Function } \textit{obj}  \gets \sum_{i=1}^{m} g(X)_{i}$
		\BState \emph{loop}:
		\If {$(\text{Degree of term i in} \; \textit{obj})  == 3$} 
		\State $C \gets C \cup (Term \; i)$
		\EndIf
		\State $Y \gets IP(C)$
		\State $M^{*} \gets LB(C,Y) \text{using Equation 3}$
		\State $Q \gets \text{Transformed QUBO using sets C, Y and} \; M^{*}$
		\State $\text{Optimize} \; Q \; \text{using Path Relinking Tabu Search routine} $
		\EndProcedure
	\end{algorithmic}
\end{algorithm}

\section{Computational Experiments} \label{tests}

For computational tests, we utilized 3-SAT instances from the DIMACS benchmark set for SAT problems (\cite{hoos2000satlib}). The number of variables ranges from 95 to 450 variables, with the number of clauses between 254 to 1680. More specifically, we focused on the flat graph coloring instances (FLAT), Random 3-SAT Backbone-minimal instances (RTI and BMS), and Random 3-SAT with Controlled Backbone Size instances (CBS). The backbone of a satisfiable Max 3-SAT instance is given by the set of literals that are true in every satisficing assignment. Backbone size is related to the difficulty of the problem instance (\cite{monasson1999determining}). The details of the studied instance types are presented in Table 1. We also present the resulting size of the Q matrix (as a result of the proposed transformation) in terms of average number of nodes and edges for each of these instance types. All these instances have maximum satisfiability, implying that all clauses are satisfiable through some variable assignment. In other words, the optimal objective function value in our case is given by $m$, where $m$ is the number of clauses. Note that this also translates as a use-case for combining goal programming with QUBO as described in \cite{verma2021goal}. Herein, we are interested in solutions wherein $max \; x'Qx = m$. We reserve such topics for future research.

\begin{table}
	\centering
	\caption{Characteristics of Benchmark Max 3-SAT Instances}
	\begin{tabular}{l|ll|ll|l|ll}
		\hline
		Instance Type & \multicolumn{2}{c}{\# Variables} & \multicolumn{2}{c}{\# Clauses} & \# Instances & \multicolumn{2}{c}{Q size} \\
		& Min             & Max            & Min           & Max            &    & Avg Nodes & Avg Edges          \\
		\hline
		FLAT          & 95              & 450            & 300           & 1680           & 1500     & 262 & 1112    \\
		RTI           & 100             & 100            & 429           & 429            & 500    & 421 & 2219      \\
		BMS           & 100             & 100            & 254           & 318            & 500    & 333 & 1592      \\
		CBS           & 100             & 100            & 429           & 429            & 25000 & 414 & 2169  \\
		\hline    
	\end{tabular}
\end{table}

For these different types of instances, we ran Algorithm 1 for a time limit of 30 seconds. The transformation routine described in Algorithm 1 was implemented in Python 3.6. The experiments were performed on a 3.40 GHz Intel Core i7 processor with 16 GB RAM running 64 bit Windows 7 OS. We present the aggregate results grouped by each instance type in Table 2. The complete set of results can be obtained from \cite{res}. For each instance type, we present ``\% Satisfied" which represents the best solution obtained by the solver in the limited runtime as a percentage of the optimal solution given by the total number of clauses. In other words, ``\% Satisfied" is given by $z_{best}*100/m$ where $z_{best}$ is the best objective function value obtained by the heuristic solver in 30 seconds. Our intention is to showcase the efficient modeling of a cubic pseudo-Boolean function in the QUBO framework. We do not aim to compete with the classical Max 3-SAT solvers. However, our choice of allotted time of 30 seconds is a reasonable estimate of the average solution time of various algorithms on benchmark instances (see \cite{ishtaiwi2021dynamic} for a recent experimental evaluation).  We observe from Table 2 that the QUBO transformation for Max 3-SAT performs very efficiently with completely determining the required variable assignments for many instances. The performance of the algorithm on flat graph coloring instances (FLAT) is exemplary, with the algorithm leading to maximum satisfiability in all instances. This resonates with the findings by the authors in \cite{kochenberger2005unconstrained} that established the efficacy of the QUBO modeling framework for graph coloring problems.

\begin{table}
	\centering
	\caption{Aggregate results of QUBO transformation on Max 3-SAT Instances}
	\begin{tabular}{l|l|llll}
		& Benchmark & Min & Max & Mean & Stdev \\
		\hline
		& Flat      & 99  & 100 & 100  & 0.2   \\
		& RTI       & 96  & 99  & 98   & 0.6   \\
		\% Satisfied & BMS       & 96  & 100 & 98   & 0.5   \\
		& CBS       & 94  & 100 & 98   & 0.6   \\
		\hline
		& Flat      & 100 & 100 & 100  & 0     \\
		& RTI       & 95  & 100 & 98   & 0.8   \\
		\% Reduction & BMS       & 96  & 100 & 99   & 0.7   \\
		& CBS       & 95  & 100 & 98   & 0.8   \\
		\hline
		& Flat      & 1   & 1   & 1    & 0     \\
		& RTI       & 2   & 4   & 3    & 0.3   \\
		$M_{LB}$        & BMS       & 2   & 4   & 3    & 0.5   \\
		& CBS       & 2   & 6   & 3    & 0.3  \\
		\hline
	\end{tabular}
\end{table}

The solution quality of Algorithm 1 is captured via the ``\% Satisfied" column. We could observe from Table 2 that the proposed efficient transformation with a generic QUBO solver consistently performs very well. The average computational performance on benchmark datasets in terms of the number of satisfied clauses exceeds 97\%. Next, we analyze the impact of Lemma 1 on the computational experiments through the ``\% Reduction" column, which represents the percentage of the quadratic subterms that are dominating in the IP used for determining the minimum number of auxiliary variables. The cubic terms involving these quadratic subterms are eliminated from consideration in the IP. Thus, such quadratic subterms are always chosen for substitution with an auxiliary variable in the cubic terms that share commonality with such quadratic subterms. The ``\% Reduction" column clearly demonstrates the efficacy of Lemma 1 which reduces the size of the IP by around 95-100\%. The naive implementation without Lemma 1 involves the complexity of $O(3m)$ for the size of the IP used in finding minimal auxiliary variables wherein $m$ is the number of clauses. Hence, the proposed Lemma 1 leads to a significant reduction in the number of variables required.

We also report the distribution of $M_{LB}$ determined by Equation 3 utilized by the QUBO transformation. Note that the lower bound ranges from $1$ to $6$, which is significantly smaller than an arbitrary choice of $M$, for example, $M=25$ or $M=100$ or $M=1000$. As demonstrated in \cite{verma2020optimal}, the choice of a lower value of $M_{LB}$ leads to a significant improvement in computational performance. In this use case, moving from $M = M_{LB}$ to $(25, 100, 1000)$ deteriorates the solution quality slightly from $Mean \; \% \; Satisfied = 97.71\%$ to $(97.59\%, 97.55\%, 97.42\%)$ for the RTI datasets. Thus, the average solution quality on the benchmark instances progressively declines as $M$ increases. Upon close investigation, $M = M_{LB}$ improves the computational performance in terms of the number of satisfied clauses in around 60\% of the instances compared to other $M$ values. Hence, the penalty coefficient $M$ can be reduced by an order of magnitude without affecting performance. Note that reducing the magnitude of the penalty coefficient $M$ is also useful for the current generation of quantum annealers (relying on QUBO models) because of the limited range and interval of the coefficient.

\section{Conclusions and Future Research}

We present a quadratic transformation of higher degree pseudo-Boolean function while efficiently minimizing the number of additional variables and the associated penalty coefficient. Extensive computational tests indicate that the approach is viable for Max 3-SAT using a generic metaheuristic designed for the QUBO framework. Further improvement in computational performance could be achieved by utilizing the techniques proposed in \cite{verma2021constraint,verma2021qubo}.

%\begin{acknowledgements}
%If you'd like to thank anyone, place your comments here
%and remove the percent signs.
%\end{acknowledgements}

% BibTeX users please use one of
%\bibliographystyle{spbasic}      % basic style, author-year citations
\bibliographystyle{spmpsci}      % mathematics and physical sciences
\bibliography{m3sat}   % name your BibTeX data base

\end{document}